\documentclass[11pt, a4paper]{amsart}

\usepackage[margin=1in]{geometry} 

\usepackage{amsmath}
\usepackage{amssymb}
\usepackage{amsthm}
\usepackage{amscd}
\usepackage{mathtools}
\usepackage{graphicx}
\usepackage{enumitem}
\usepackage{faktor}
\usepackage{MnSymbol} 

\usepackage{marginnote}
\usepackage{todonotes}
\usepackage{bbm}
\usepackage{bm}
\usepackage[colorlinks=true,linkcolor=black,citecolor=black,urlcolor=black]{hyperref} 

\usepackage{tikz-cd} 
\tikzcdset{scale cd/.style={every label/.append style={scale=#1}, cells={nodes={scale=#1}}}}
\usetikzlibrary{arrows}

\usepackage[style=alphabetic]{biblatex} 
\renewbibmacro{in:}{} 

\newtheorem{theorem}{Theorem}[section]
\newtheorem*{theorem*}{Theorem}

\newtheorem*{conjecture*}{Conjecture}

\newtheorem*{question*}{Question}

\newtheorem*{guess*}{Guess}
\newtheorem*{Assumption*}{Assumption}

\newtheorem*{problem*}{Problem}
\newtheorem{lemma}[theorem]{Lemma}
\newtheorem*{lemma*}{Lemma}
\newtheorem{proposition}[theorem]{Proposition}
\newtheorem*{proposition*}{Proposition}
\newtheorem{corollary}[theorem]{Corollary}
\newtheorem*{corollary*}{Corollary}

\theoremstyle{definition}

\newtheorem*{definition*}{Definition}

\newtheorem{remark}[theorem]{Remark}
\newtheorem*{remark*}{Remark}

\newtheorem{example}[theorem]{Example}

\newtheorem*{example*}{Example}
\newtheorem*{examples*}{Examples}

\newcommand{\onto}{\twoheadrightarrow}

\newcommand{\Hom}{\operatorname{Hom}}

\newcommand{\Res}{\operatorname{Res}}

\thanks{The author acknowledges funding from the Heilbronn Institute for Mathematical Research (HIMR) and the UK Engineering and Physical Sciences Research Council under the $\lq\lq$Additional Funding Programme for Mathematical Sciences". (Grant number: EP/V521917/1).}

\title{Contramodules for algebraic groups: induction and projective covers}
\author{Dylan Johnston}
\date{}

\begin{document}

\begin{abstract} 
\noindent{}In this paper we investigate contramodules for algebraic groups. Namely, we give contramodule analogs to two $20^{\textup{th}}$ century results about comodules. Firstly, we show that induction of contramodules over coordinate rings of algebraic groups is exact if and only if the associated quotient variety is affine. Secondly, we give an inverse limit theorem for constructing projective covers of simple $G$-modules using $G$-structures of projective covers of simple modules of the first Frobenius kernel, $G_1$. We conclude by showing that the inverse limit theorem is a special case of a more general phenomenon between injective covers in $k[G]$-Comod and projective covers in $k[G]$-Contra.
\end{abstract}

\maketitle

\section*{Introduction.}

Contramodules were introduced in a 1965 paper of Eilenberg and Moore alongside comodules \cite{eilenberg1965foundations}. However, they were largely neglected until the early 2000s, when Positselski mentioned them in a series of letters. Since then, research into contramodules has increased considerably, with special mention going once again to Positselski \cite{positselski2021contramodules}, 
 \cite{positselski2010homological}.

Let $(C, \Delta )$ be a coalgebra over a field $k$, then, as well as comodules, we have other $\lq\lq$module-like structures" on $C$, called contramodules. A $C$-contramodule is a $k$-vector space $B$ along with a linear map $\theta_B: \Hom_k(C,B) \longrightarrow B$ satisfying some compatibility conditions which we make explicit in the next section. Given any vector space $V$ we can form the free $C$-contramodule $\Hom_k(C,V)$, with the linear map $\Hom_k\big(C,\Hom_k(C,V)\big) \cong \Hom_k(C \otimes C, V) \longrightarrow \Hom_k(C,V)$ being given by comultiplication on $C$ in the first factor. We will see in the next section that contramodules of this form are projective, and so, in particular, the category $C$-Contra has enough projectives.

Furthermore, one may also speak of contramodules over a $k$-group scheme $G$. Let $k[G]$ be the coordinate ring of $G$, that is, $G(A) = \Hom_{k\text{-alg}}(k[G],A)$ for all $k$-algebras $A$. Then, it is well known that a $k[G]$-comodule $M$ is equivalent to a $k$-vector space $M$ along with an action of $G(A)$ on $M \otimes_k A$ for each $k$-algebra $A$ \cite[Chapter I.2.8]{jantzen2003representations}. One can similarly show that a $k[G]$-contramodule $M$ is equivalent to a $k$-vector space along with an action of $G(A)$ on $\Hom_k(A,M)$ for each $k$-algebra $A$.

In this paper, we state and prove two analogous results for contramodules that were originally stated in terms of comodules. First, we give the analog to the main result of Cline, Parshall, and Scott \cite{cline1977induced}. That is, we prove the following result.

\begin{theorem}
Let $H$ be a closed subgroup of an algebraic group $G$ with coordinate rings $k[H]$ and $k[G]$, respectively. Then the following are equivalent:
\begin{enumerate}[label=\roman*)]
    \item $G/H$ is an affine variety
    \item $H$ is contra-exact in $G$, that is, $\text{Ind}\,_{k[H]}^{k[G]}: k[H]$-Contra $\longrightarrow k[G]$-Contra is exact
    \item $H$ is exact in $G$, that is, $\text{ind}\,_{H}^G: H$-Mod $\longrightarrow G$-Mod is exact.
\end{enumerate}
\end{theorem}

This result shows that, just like comodules, contramodules of algebraic groups can detect geometric properties of the algebraic groups. Here, in particular, we see that the exactness of the contra-induction functor can detect the affineness of the quotient variety.

Our second result is an inverse limit theorem. This provides a way to construct projective covers of simple $G$-modules as an inverse limit of certain twisted tensor products of $G$-structures of projective covers of simple $G_1$-modules. As previously mentioned, we will see that the category of $k[G]$-contramodules has enough projectives, suggesting that this is the category we should work in to construct these covers. The result uses ideas from a 1980 note by Donkin, where a short proof was given of the direct limit theorem \cite{donkin1980}. Our result is the following: let $G$ be a simply connected semisimple algebraic group over an algebraically closed field $k$ with char$(k) = p$. For each dominant weight $\lambda$, let $L(\lambda)$ denote the simple $G$-module with highest weight $\lambda$. Let $G_r$ denote the $r^{th}$ Frobenius kernel of $G$.  Let $X_1$ be the set of $p$-restricted weights, and for each $\lambda \in X_1$ let $L_1(\lambda)$ denote the simple $G_1$ module with highest weight $\lambda$. Finally, let $P_1(\lambda)$ be the $G_1$-projective cover of $L_1(\lambda)$. Suppose that $P(\lambda)$ is a $G$-structure on $P_1(\lambda)$ for each $\lambda \in X_1$ (note that to date the existence of such a structure is only known in general for $p \geq 2h-4$, where $h$ is the Coxeter number of $G$ \cite{bendel2023donkins}).
\begin{theorem}
Let $\lambda$ be a dominant weight with $p$-adic expansion $\displaystyle \lambda = \sum_{i=0}^s \lambda_i\,p^i$ with each $\lambda_i \in X_1$. For $r>s$ define
\[ P_{\lambda,r} = P(\lambda_0) \otimes P(\lambda_1)^{(1)} \otimes \dots \otimes P(\lambda_s)^{(s)} \otimes P(0)^{(s+1)} \otimes \dots \otimes P(0)^{(r-1)}\] 
where the superscript $(i)$ denotes a twist by $i$ applications of the Frobenius morphism. Fix a projection $q: P(0) \longrightarrow L(0)$ and define $P_{\lambda} := \varprojlim P_{\lambda,r}$ in the category $k[G]$-Contra. Then $P_{\lambda}$ is the projective cover of the simple $G$-module, $L(\lambda)$, considered as a $k[G]$-contramodule.
\end{theorem}
This result gives us an explicit description of projective covers in the category $k[G]$-Contra. Moreover, it shows that taking limits of finite-dimensional modules is a contramodule, that is, if we work in the category $k[G]$-Contra, then we have control when taking limits. 

Finally, we prove the following result, showing that the above theorem is a special case of a more general phenomenon:

\begin{theorem}
Let $G$ be a simply connected semisimple algebraic group. Let $I(\lambda)$ be the rationally injective $G$-envelope of $L(\lambda)$. Suppose $I(\lambda) = \varinjlim_i X_i$, where the $X_i$ are finite-dimensional $G$-modules. Then $P(\lambda^*)$, the $k[G]$-Contra projective cover of $L(\lambda)^*$, the dual module of $L(\lambda)$, is isomorphic to $\varprojlim_i X_i^*$, as contramodules.    
\end{theorem}

The paper is organised into four main sections. In Section \ref{Section-contra-defns}, we begin with a short review of coalgebras and comodules. Then, we give the definition of a contramodule, and finally, provide some important constructions that will be used later in the paper. In Section \ref{Section-res-and-ind}, we give the construction of the induction functor for contramodules and prove that induction is left adjoint to the restriction functor. In Section \ref{Section-affine}, we state and prove our first main result, which relates the affineness of a quotient variety of two algebraic groups to the exactness of the contra-induction functor. Finally, in Section \ref{Section-inverse-limit}, using constructions introduced in Section \ref{Section-contra-defns}, we state and prove the two results regarding the construction of the projective covers of simple $G$-modules in the category $k[G]$-Contra.

\section*{Acknowledgements.}

I would like to thank my supervisor Professor Dmitriy Rumynin for introducing me to this topic, and for his continuous patience and guidance throughout. I would also like to thank Professor Stephen Donkin for suggesting that a more general relationship between injective covers and projective covers (c.f. Theorem \ref{Donkinconj}) may be occurring.

\section{Contramodules: definitions, examples and constructions.}\label{Section-contra-defns}

We begin by recalling the definitions of a coalgebra over a field and a comodule over a coalgebra \cite{sweedlerHopfAlgebras}. Let $k$ be a field. A coalgebra $(C,\Delta,\epsilon)$ is a $k$-vector space $C$ along with a linear comultiplication map $\displaystyle\Delta: C \longrightarrow C \otimes_k C$ and linear counit map $\epsilon: C \longrightarrow k$ which satisfy coassociativity and counity conditions. That is, the following diagrams commute:
\[
 \begin{tikzcd}[scale cd=1.1]
{C} \arrow[dd, "{\Delta}"] \arrow[rr, "{\Delta}"]  & & {C \otimes C} \arrow[dd, "\text{Id}_C \otimes \Delta"] \\
     & &       \\
    {C \otimes C} \arrow[rr, "\Delta \otimes \text{Id}_C"]  & & C \otimes C \otimes C
\end{tikzcd}
\hspace{10ex}
\begin{tikzcd}[scale cd=1.1]
{C} \arrow[ddrr, "{\cong}"] \arrow[rr, "{\Delta}"] \arrow[dd, "{\Delta}"]  & & {C \otimes C} \arrow[dd, "\text{Id}_C \otimes \epsilon"] \\
     & &       \\
    {C \otimes C} \arrow[rr, "\epsilon \otimes \text{Id}_C"]  & & k \otimes C \cong C \otimes k.
\end{tikzcd}\]
Recall that it is common to write the comultiplication of a coalgebra using Sweedler notation. That is, $\Delta(c) = \sum_{(c)}c_{(1)} \otimes c_{(2)}$, with coassociativity enabling us to write $\big((\text{Id}_C \otimes \Delta)\circ\Delta\big)(c) = \big((\Delta \otimes \text{Id}_C \Delta)\circ\Delta\big)(c) = \sum_{(c)} c_{(1)} \otimes c_{(2)} \otimes c_{(3)}.$ We will also make use of this notation throughout. A (right) $C$-comodule $(M,\Delta_M)$ is a $k$-vector space $M$ with a linear map $\Delta_M: M \longrightarrow M \otimes_k C$, called the coaction, such that the following diagrams commute:

\[
 \begin{tikzcd}[scale cd=1.1]
{M} \arrow[dd, "{\Delta_M}"] \arrow[rr, "{\Delta_M}"]  & & {M \otimes C} \arrow[dd, "\Delta_M \otimes \text{Id}_C"] \\
     & &       \\
    {M \otimes C} \arrow[rr, "\text{Id}_M \otimes \Delta"]  & & M \otimes C \otimes C \\
\end{tikzcd}
\hspace{10ex}
\begin{tikzcd}[scale cd=1.1]
{M} \arrow[ddrr, "{\cong}"] \arrow[rr, "{\Delta_M}"]& & {M \otimes C} \arrow[dd, "\text{Id}_M \otimes \epsilon"] \\
     & &       \\
     & & M \otimes k. \\
\end{tikzcd}\]
In Sweedler notation, the coaction map can be written as $\Delta_M(m) = \sum_{(m)} m_{(0)} \otimes m_{(1)}$ for any $m \in M$. One can also define left $C$-comodules in the obvious way. Given two right $C$-comodules, $M$ and $N$, let $\Hom_C(M,N)$ denote the space of comodule homomorphisms between $M$ and $N$, that is, the linear maps $f: M \longrightarrow N$ such that the following diagram commutes:
\[
 \begin{tikzcd}[scale cd=1.1]
{M} \arrow[dd, "{f}"] \arrow[rr, "{\Delta_M}"]  & & {M \otimes C} \arrow[dd, "f\otimes \text{Id}_C "] \\
     & &       \\
    {N} \arrow[rr, "\Delta_N"]  & & N \otimes C.
\end{tikzcd}
\]
Observe that $C$ is itself a $C$-comodule with coaction $\Delta_C := \Delta$. More generally, given any vector space $V$, define a comodule over $C$ given by the vector space $V \otimes_k C$ and coaction $\Delta_{V \otimes_k C} := \text{Id}_V \otimes \Delta$. We call such comodules the \textit{cofree comodules}, and one may verify that for any (right) $C$-comodule $W$ we have
\[\displaystyle\Hom_C(W, V \otimes_k C) \,\cong\,\Hom_k(W,V)\]
as $k$-vector spaces, where $\Hom_k(W,V)$ denotes the linear maps from $W$ to $V$. Therefore, the cofree comodules are injective, and it follows that a comodule $M$ is injective if and only if it is a direct summand of a cofree comodule.

Now, given a right $C$-comodule M and a left $C$-comodule $N$, we can form their cotensor product $M \boxtimes_C N$ which is a $k$-vector subspace of $M \otimes_k N$ given by the equaliser of $\Delta_M \otimes \text{Id}_N$ and $\text{Id}_M \otimes \Delta_N$, that is, the set of elements $a \in M \otimes_k N$ such that $(\Delta_M \otimes \text{Id}_N)(a) = (\text{Id}_M \otimes \Delta_N)(a):$
\[ 
\begin{tikzcd}[scale cd=1, sep = huge]
M \boxtimes_C N = \text{eq}\bigg(M \otimes N \ar[r,shift left=.75ex,"\Delta_M \otimes \text{Id}_N"]
  \ar[r,shift right=.75ex,swap,"\text{Id}_M \otimes \Delta_N"]
&
M \otimes C \otimes N \bigg).\end{tikzcd}\]
If $V$ is a finite-dimensional right $C$-comodule, then $V^*$ has a natural structure of a left $C$-comodule, and we have the following isomorphism of vector spaces for any right $C$-comodule $M$:
\[ \Hom_C(V,M) \,\cong\, M \boxtimes_C V^*.\]

We now turn our attention to contramodules over a coalgebra $C$. A (right) $C$-contramodule $(B,\theta_B)$ is a $k$-vector space $B$ equipped with a linear map $\theta_B : \Hom_k(C,B) \longrightarrow B$, called the contra-action, satisfying contra-associativity and contra-unity conditions. That is, the following two diagrams commute:

\[\hspace{-0.75ex}
 \begin{tikzcd}[scale cd=0.88]
{\Hom_k\big(C,\Hom_k(C,B)\big)} \arrow[dd, "{\otimes\, \dashv\, \Hom}"] \arrow[rr, "{\Hom_k(C,\theta_B)}"]  & & {\Hom_k(C,B)} \arrow[dd, "\theta_B"] \\
     & &       \\
    {\Hom_k(C \otimes C,B)} \arrow[r, "{\Hom_k(\Delta,B)}"]  & {\Hom_k(C,B)} \arrow[r, "{\theta_B}"] & B
\end{tikzcd}
\hspace{1.5ex}
\begin{tikzcd}[scale cd=0.88]
{\Hom_k(k,B)} \arrow[ddrr, "{\cong}"] \arrow[rr, "{\Hom_k(\epsilon,B)}"]  & & {\Hom_k(C,B)}  \arrow[dd, "{\theta_B}"] \\
     & &       \\
    & & B
\end{tikzcd}\]
where $\lq\lq \otimes \dashv \Hom"$ denotes the tensor-hom adjunction, which for any vector spaces $U,V,W$ is given by identifying $\Hom_k\big(U, \Hom_k(V,W)\big)$ and $\Hom_k\left(U \otimes_k V, W\right)$. We remark that using instead the identification $\Hom_k\big(U, \Hom_k(V,W)\big) \cong \Hom_k(V \otimes_k U, W)$ gives the definition of a left $C$-contramodule. Unless stated otherwise all contramodules will be right contramodules.

Given two right $C$-contramodules $B$ and $D$, let $\Hom^C(B,D)$ denote the space of contramodule homomorphisms from $B$ to $D$, that is, the linear maps $f: B \longrightarrow D$ such that the following diagram commutes: 
\[
 \begin{tikzcd}[scale cd=1]
{\Hom_k(C,B)} \arrow[dd, "{\Hom_k(C,f})"] \arrow[rr, "{\theta_B}"]  & & {B} \arrow[dd, "f"] \\
     & &       \\
    {\Hom_k(C,D)} \arrow[rr, "{\theta_D}"]  && D.
\end{tikzcd}\]
Now consider the vector space $\Hom_k(C,k)$. It can be given the structure of a $C$-contramodule by applying the comultiplication of $C$ in the first factor. That is, the contramodule structure on $\Hom_k(C,k)$ is given by the composition
\[\Hom_k(C,\Hom_k(C,k)) \cong \Hom_k(C \otimes C, k) \xrightarrow{\Hom_k(\Delta,k)} \Hom_k(C,k).\]
More generally, one may replace $k$ with any vector space $V$, obtaining what are called the $\textit{free contramodules}$. One can show that there is an isomorphism of vector spaces
\[ \Hom^C\big(\Hom_k(C,V),W\big) \cong \Hom_k(V,W)\]
 for every (right) $C$-contramodule $W$. In particular, free contramodules are projective. It follows that any contramodule $B$ is projective if and only if it is a direct summand of a free contramodule. It is worth noting that so far, the situation for contramodules is, in some sense, dual to that for comodules.

The above example can be generalised further by replacing $C$ with a left $C$-comodule $M$. Then $\Hom_k(M,V)$ can be given a right $C$-contramodule structure in much the same way, explicitly, the structure map is:
\[ \Hom_k(C,\Hom_k(M,V)) \cong \Hom_k(C \otimes M,V) \xrightarrow{\Hom_k(\Delta_M,V)} \Hom_k(M,V).\]
Note that considering right comodules instead produces left contramodules. To conclude the section, we introduce two constructions which will be heavily used in the later sections. Firstly, given a left $C$-comodule M and a right $C$-contramodule $B$, let $B \odot_C M$ denote the contratensor product of $B$ and $M$. It is a quotient vector space of $B \otimes_k M$ given by the following coequaliser: 
\[ 
\begin{tikzcd}[scale cd=1, sep = huge]
B \odot_C M = \text{coeq}\bigg(\Hom_k(C,B) \otimes_k M \ar[r,shift left=.75ex,"\theta_B \otimes \text{Id}_M"]
  \ar[r,shift right=.75ex,swap,"(\text{eval}) \circ \Delta_M"]
&
B \otimes_k M \bigg),\end{tikzcd}\]
that is, $B \odot_C M = (B \otimes_k M)\big/\big(\text{im}( \theta_B \otimes \text{Id}_M - (\text{eval}) \circ \Delta_M)\big)$ where $(\text{eval}) \circ \Delta_M$ denotes the composition
\[ \Hom_k(C,B) \otimes_k M \xrightarrow{\text{Id}_{\Hom_k(C,B)} \otimes \Delta_M} \Hom_k(C,B) \otimes_k C \otimes_k M  \xrightarrow{\text{eval} \otimes \text{Id}_M} B \otimes_k M.\]
Our final construction involves a right $C$-comodule $M$ and a right $C$-contramodule $B$. Let $\text{Cohom}_C(M,B)$ denote the cohomomorphisms between $M$ and $B$. It is a quotient vector space of $\Hom_k(M,B)$ given by the following coequaliser:
\[ 
\begin{tikzcd}[scale cd=1, sep = huge]
\text{Cohom}_C(M,B) = \text{coeq}\bigg(\Hom_k(M \otimes C,B) \ar[r,shift left=.75ex,"{\Hom_k(\Delta_{M},B)}"]
  \ar[r,shift right=.75ex,swap,"{\Hom_k(M,\theta_B)}"]
&
\Hom_k(M,B)\bigg)\end{tikzcd}\]
where $\Hom_k(M,\theta_B) : \Hom_k(M \otimes C,B)\, \cong\, \Hom_k\big(M, \Hom_k(C,B)\big)  \longrightarrow \Hom_k(M,B).$

Note that if $M$ is a finite-dimensional right $C$-comodule and $B$ is any right $C$-contramodule then these two constructions are heavily related \cite[Section 3.1]{positselski2021contramodules}. In particular, we have
\[ \text{Cohom}_C(M,B) \cong B \odot_C M^*.\]

So far, then, we have seen that these two `module-like structures' over $C$ have some interplay between them. A particular case of this (and a case which will be very useful to us later) is given a finite-dimensional $C$-comodule we may equip it with a $C$-contramodule structure. In fact, this will give us a functor from $C$-Comod$_{\text{f.d.}}$ to $C$-Contra. We make this explicit in the next lemma.

\begin{lemma}\label{f.d.-comod-to-contra}
Let $C$ be a coalgebra over a field $k$ and let $(M,\Delta_M)$ be a finite-dimensional right $C$-comodule. Then $M$ has a right contramodule structure given by the following composition: 
\begin{align*}
\Hom_k(C,M) \cong C^* \otimes M \xrightarrow{\textup{Id}_{C^*} \otimes \Delta_M} C^* \otimes M \otimes C \xrightarrow{eval} M
\end{align*}
where the first isomorphism holds due to the finite-dimensionally of $M$, and where the map `eval' is the obvious evaluation map involving $C^*$ and $C$. Moreover, this defines a functor from $C$-Comod$_{\text{f.d.}}$ to $C$-Contra.
\end{lemma}

\begin{proof}
For ease of notation, denote the claimed contramodule structure on $M$ by $\theta_M$. To prove contra-associativity we utilise the finite-dimensionally of $M$ to fix a basis $\{m_i\}$. We need to show that
$\theta_M \circ \Hom_k(C,\theta_M) \equiv \theta_M \circ \Hom_k(\Delta,M) \circ (\otimes\, \dashv\, \Hom): \Hom_k(C,\Hom_k(C,M)) \longrightarrow M $.

Let $\displaystyle \big(c \mapsto (d \mapsto \sum_i \alpha_{c,d}^i m_i)\big) \in \Hom_k(C,\Hom_k(C,M)).$ Applying the left hand map results in $\displaystyle \sum_i \alpha_{m_{i_{(0)_{(1)}}},m_{i_{(1)}}}^im_{i_{(0)}}$. On the other hand, applying the right map gives us $\displaystyle \sum_i \alpha_{m_{i_{(1)_{(1)}}},m_{i_{(1)_{(2)}}}}^im_{i_{(0)}}$ with equality following from coassociativity of $M$ (as a comodule). 

Showing contra-unity is similar, with equality of the two expressions following from counity of $M$ (as a comodule).

Finally, to check that this construction defines a functor we need to check that given a comodule homomorphism $f: (M,\Delta_M) \longrightarrow (N,\Delta_N)$, the map $f: (M,\theta_M) \longrightarrow (N,\theta_N)$ is a contramodule homomorphism. In other words, given that $f$ satisfies $\Delta_N \circ f \equiv (f \otimes \text{Id}_C) \circ \Delta_M$, we must show that $f$ satisfies $\theta_N \circ \Hom_k(C,f) \equiv f \circ \theta_M$.

Once again we make use of our fixed basis $\{m_i\}$ of $M$ to write an arbitrary element of $\Hom_k(C,M)$ in the form $(c \mapsto \sum_i \alpha_i(c) m_i)$ for some $\alpha_i \in C^*$. Applying $\theta_N \circ \Hom_k(C,f)$ yields $\displaystyle \sum_i \alpha_i(f(m_i)_{(1)}) f(m_i)_{(0)}$. Alternatively, applying $f \circ \theta_M$ gives $\displaystyle \sum_i \alpha_i(m_{i_{(1)}}) f(m_{i_{(0)}})$, but these two expressions are equal precisely because $f$ is a comodule homomorphism.
\end{proof}

To finish the section, we briefly discuss the case where the coalgebra is the coordinate ring, $k[G]$, of an algebraic group $G$ over a field $k$. To avoid any confusion we first get our lefts and rights in order. Modules over the group $G$ will be left modules, and comodules and contramodules over the coordinate ring $k[G]$ will be right comodules and contramodules. 

Recall that a $k[G]$-comodule is equivalent to a $k$-module $M$ along with an action of $G$ on $M_a$, that is, an action of $G(A)$ on $M_a(A) := M \otimes_k A$ via $A$-linear maps for all $k$-algebras $A$ \cite[I.2.8]{jantzen2003representations}. It should be noted that this along with Lemma \ref{f.d.-comod-to-contra} allows us to put a $k[G]$-contramodule structure on finite-dimensional $G$-modules. We have an analogous way to describe contramodules over $k[G]$. Namely, we have the following:

\begin{lemma}
A $k[G]$-contramodule structure on $M$ is equivalent to a natural action of $G$ on $M_{\textup{Hom}} := \Hom_k(-,M)$, that is, an action of $G(A)$ on $\Hom_k(A,M)$ by $A$-linear maps for all $k$-algebras $A$ which is compatible with $k$-algebra homomorphisms $A \rightarrow A'$.
\end{lemma}

\begin{proof}
Let $\theta: \Hom_k(k[G],M) \rightarrow M$ be a contramodule structure on $M$. For all $k$-algebras $A$ define 
\begin{align*}
    \bullet_A : G(A) \times \Hom_k(A,M) &\longrightarrow \Hom_k(A,M) \\
    (g,\, \varphi) &\longmapsto \bigg(a \mapsto \theta\Big(c \mapsto \varphi\big(ag(c)\big)\Big)\bigg).
\end{align*}
We claim that for all $k$-algebras $A$, this defines a $G(A)$ action on $\Hom_k(A,M)$. Indeed, let $\mathbbm{1}: k[G] \rightarrow A \in G(A)$ denote the unit, explicitly we have, $\mathbbm{1}(f) = \epsilon(f)\cdot1_A$ for all $f \in k[G]$, where $\epsilon$ is the counit map. Then we have:
\begin{align*}
    \mathbbm{1} \bullet_A \varphi & \equiv \bigg(a \mapsto \theta\Big(c \mapsto \varphi\big(a\mathbbm{1}(c)\big)\Big)\bigg) \equiv \bigg(a \mapsto \theta\Big(c \mapsto \epsilon(c)\varphi\big(a)\big)\Big)\bigg) \\
    & \equiv \Big(a \mapsto \varphi(a)\Big) \equiv \varphi
\end{align*}
with the penultimate equivalence holding due to the contra-unity of $\theta$. Now let $g,h \in G(A)$, $\varphi \in \Hom_k(A,M)$, then we may compute $g \bullet h \bullet \varphi$ in two ways. On the one hand we have 
\begin{align*}
    g \bullet (h \bullet \varphi) &\equiv g \bullet \bigg(a \mapsto \theta\Big(c \mapsto \varphi\big(ah(c)\big)\Big)\bigg) \\
    &\equiv  \bigg(a \mapsto \theta\Big(c \mapsto \theta\Big(d \mapsto \varphi\big(ag(c)h(d)\big)\Big)\Big)\bigg).
\end{align*}
Now, recall that given $g,h \in G(A) = \Hom_{k-alg}(k[G],A)$ we have $(g \cdot h)(c) = g(c_{(1)})h(c_{(2)})$ where we write $\displaystyle\Delta(c) = c_{(1)} \otimes c_{(2)}$ (the sum is implicit). Therefore, on the other hand, we have
\begin{align*}
(g \cdot h) \bullet \varphi &\equiv \bigg(a \mapsto \theta\Big(c \mapsto \varphi\big(ag(c_{(1)})h(c_{(2)})\big)\Big)\bigg).   
\end{align*}
Applying the contra-associativity of $\theta$ shows that these two expressions are equal. Moreover, the naturality of this action is straightforward to check.

Conversely, for each $k$-algebra $A$, let $\bullet_A: G(A) \times \Hom_k(A,M) \rightarrow \Hom_k(A,M)$ be an action. By the naturality of the action, we have the following commutative diagram for any $k$-algebra $A$ and $g \in G(A)$: 
\[
 \begin{tikzcd}[scale cd=1]
{\Hom_k(k[G],M)} \arrow[rr, "{\text{id}_{k[G]}\, \bullet_{k[G]}} \,-"]  & & {\Hom_k(k[G],M)} \\
     & &       \\
    {\Hom_k(A,M)} \arrow[uu, "{\Hom_k(g,M})"]  \arrow[rr, "{g\, \bullet_A\,-}"]  && {\Hom_k(A,M)} \arrow[uu, "{\Hom_k(g,M)}"] .
\end{tikzcd}\]
Explicitly, for any $\varphi \in \Hom_k(A,M)$ we have 
\begin{equation}\label{naturality}\tag{$*$}
\text{id}_{k[G]} \bullet_{k[G]} (\varphi \circ g) \equiv (g \bullet_A \varphi) \circ g.   
\end{equation} 
We claim that the following composition is a $k[G]$-contramodule structure on $M$:
\begin{align*}
    \theta: \Hom_k(k[G],M) \xrightarrow{\,\,\,\text{id}_{k[G]} \bullet\, -\,\,\,} \Hom_k(k[G],M) \xrightarrow{\,\,\,\text{evaluation at }1 \in k[G]\,\,\,} M.
\end{align*}
We will denote this composition as $(\text{id}_{k[G]} \bullet_{k[G]} -)\big|_{1}$. Rather than check the contramodule conditions directly, we instead show that the operation that produces a $G$-action from a $k[G]$-contramodule structure and vice versa are inverse to each other. Then the contramodule conditions will be satisfied by $\theta$, precisely because $\bullet_A$ is a group action, for each $k$-algebra $A$.

Thus, let's start with a contramodule structure $\theta: \Hom_k(k[G],M) \longrightarrow M$. Producing a group action from this and then converting back to a contramodule structure results in the following:
\begin{align*}
    \theta' : &\Hom_k(k[G],M) \longrightarrow M \\
    &\varphi \longmapsto\,\, \big(\text{id}_{k[G]} \bullet_{k[G]} \varphi\big)\Big|_{1} \equiv \bigg(a \mapsto \theta\Big(c \mapsto \varphi\big(a \cdot \text{id}_{k[G]}(c)\big)\Big)\bigg)\bigg|_{a=1} \\
    &\hspace{5ex}\equiv \theta\Big(c \mapsto \varphi(c)\Big)\, \equiv\, \theta(\varphi).
\end{align*}
Conversely, starting with a group action, $\bullet$, of the algebraic group $G$, producing a contramodule structure and then converting back to a group action results in the following:
\begin{align*}
G(A) \times \Hom_k(A,M) &\longrightarrow \Hom_k(A,M) \\
(g, \varphi) &\longmapsto \bigg(a \mapsto \text{id}_{k[G]} \bullet \Big(c \mapsto \varphi\big(a g(c)\big)\Big)\Big|_{c=1}\bigg)\bigg) \\
& \equiv \bigg(a \mapsto \text{id}_{k[G]} \bullet \Big(c \mapsto \big((\varphi \circ a) \circ g\big)(c)\big)\Big)\Big|_{c=1}\bigg)\bigg).
\end{align*}
Now applying equation (\ref{naturality}) along with the fact that the action of $G(A)$ is $A$-linear (that is, $(g \bullet_A (\varphi \circ a))(b) = (g \bullet_A \varphi)(ab)$ for all $g \in G(A), \varphi \in \Hom_k(A,M)$, and $a,b, \in A$, where by a slight abuse of notation the linear endomorphism of $A$ given by left multiplication by $a \in A$ is also denoted by $a$), as well as the fact that $g(1) = 1$ gives us
\[(g, \varphi) \longmapsto \bigg(a \mapsto \Big(g \bullet_A (\varphi \circ a)\Big) \circ g \big|_{1}\bigg) \equiv \bigg(a \mapsto \Big(g \bullet_A (\varphi 
 \circ a)\Big)\big(1\big)\bigg) \equiv \Big(a \mapsto \big(g \bullet_A \varphi\big)\big(a\big)\Big) \equiv g \bullet_A \varphi.\]
\end{proof}

\section{Restriction and induction of contramodules.}\label{Section-res-and-ind}

In this section, we define the induction and restriction functors for contramodules and show that they form an adjoint pair. Let $C$ and $D$ be coalgebras, and let $ \rho:C \rightarrow D$ be a coalgebra map. Let $(V,\theta)$ be a $C$-contramodule, then $V$ can be viewed as a $D$-contramodule via the following composition:
\[\begin{tikzcd}
    \Hom_k(D,V) \arrow[rr, "{\Hom_k(\rho , V)}"]&& \Hom_k(C,V) \arrow[r,"\theta"]& V. \end{tikzcd}\]
This is the restriction functor from $C$-contramodules to $D$-contramodules. We denote this by $\Res_D^C(-)$ or $(-)|_D$ (with the omission of the map $\rho$) and will often make use of a slight abuse of notation and write $\theta$ instead of $\theta \circ \Hom_k(\rho , V)$ for the composition. 

Now, let $(W,\pi)$ be a $D$-contramodule, we may construct a $C$-contramodule as follows. To begin, observe that $C$ naturally has the structure of a $D$-comodule, with the coaction given by the composition of the comultiplication of $C$ followed by the map$\begin{tikzcd}[sep =small] \textup{Id}_{C} \otimes \rho : C \otimes C \arrow[r]& C \otimes D.\end{tikzcd}$ 

Now consider the following two maps from $\Hom_k(C \otimes D, W)$ to $\Hom_k(C, W)$:
\begin{equation}\label{cohom_defining_maps}
\begin{aligned}
f:& \Hom_k(C \otimes D, W) \xrightarrow{\Hom_k\big((\textup{Id}_{C} \otimes \rho)\circ \Delta,W\big)} \Hom_k(C, W),\\
g:& \Hom_k(C \otimes D, W) \xrightarrow{\,\,\,\cong\,\,\,} \Hom_k\big(C,\Hom_k(D,W)\big) \xrightarrow{\Hom_k(C,\pi)} \Hom_k(C, W),
\end{aligned}    
\end{equation}
where in the first map we have used the right $D$-comodule structure on $C$, and in the second we have used the $D$-contramodule structure on $W$. As seen in Section \ref{Section-contra-defns}, we have the $k$-vector space \[\text{Cohom}_{D}(C,W) := \dfrac{\Hom_k(C,W)}{\text{Im}(f - g)}.\]
Now, by viewing $C$ as a left comodule over itself (via the comultiplication of $C$) we may give $\text{Cohom}_{D}(C,W)$ the structure of a right $C$-contramodule. We do this by observing that it is nothing more than a quotient of the free (right) $C$-contramodule $\Hom_k(C, W)$. This construction gives us an induction functor from $D$-contramodules to $C$-contramodules. We write the induction functor as 
\[\textup{Ind}_{D}^C : D\text{-Contra} \longrightarrow C\text{-Contra},\]
(once again omitting the map $\rho$ from our notation). Observe that if $F: C$-Contra $\longrightarrow $ Vect denotes the forgetful functor, then we have $\text{Cohom}_D(C,-) = F \circ \text{Ind}_D^C(-)$. 

Just as in the case of comodules, induction and restriction of contramodules are adjoint to each other \footnote{However, for comodules the (co)induction functor is right adjoint. Whereas for contramodules the induction functor is left adjoint.}. Explicitly we have:
\begin{proposition}
Let $\rho: C \rightarrow D$ be a map of coalgebras. Then the induction functor (defined using $\rho$) is left adjoint to the restriction functor (defined using $\rho$). That is, for all $D$-contramodules $W$ and $C$-contramodules $V$ we have
\[\Hom^{C}\Big(\textup{Ind}_{D}^{C}(W), V\Big)\, \cong\, \Hom^{D}\big(W, V|_{D}\big).\]    
\end{proposition}
 
\begin{proof}
To begin, recall that we have the following natural isomorphism of vector spaces:
\begin{align*}
    \Gamma: \textup{Hom}^{C}\big(\textup{Hom}_k(C,W),V\big) &\longrightarrow \textup{Hom}_k(W,V) \\
    \varphi &\longmapsto \Big(w \mapsto \varphi\big(c \mapsto \epsilon(c)w\big)\Big) \\
    \theta_V \circ \textup{Hom}_k(C,\psi\big) &\longleftarrow{\hspace{-6.8pt}\shortmid}\, \psi.
\end{align*}
Now, letting $f,g$ denote the previously described maps from $\Hom_k(C \otimes D, W) \longrightarrow \Hom_k(C, W)$ in equation (\ref{cohom_defining_maps}), we have the identification 
\begin{equation}\label{ident}
\Hom^{C}\big(\textup{Ind}_{D}^{C}(W), V\big) =\Big\{ \varphi \in \textup{Hom}^{C}\big(\Hom_k(C,W),V\big) :\,\, \varphi \circ (f - g) \equiv 0\Big\}.    
\end{equation}
Next, it will be convenient to have the definition of what it means to be a contramodule homomorphism at hand. Recall that if ${\varphi \in \textup{Hom}^{C}\big(\Hom_k(C,W),V\big)}$ then we have that the following diagram commutes:
\begin{equation}\tag{$\dagger$}\label{commute1}
    \begin{tikzcd}
{\Hom_k\Big(C,\Hom_k\big(C,W\big)\Big)} \arrow[dd, "{\Hom_k\big(C,\varphi\big)}"] \arrow[rrr, "{\theta_{\Hom_k(C, W)}}"] &  &  & {\Hom_k\big(C,W\big)} \arrow[dd, "\varphi"] \\
     & & &      \\
    {\Hom_k\big(C,V\big)} \arrow[rrr, "\theta_{V}"] &  & & V.
    \end{tikzcd}
\end{equation}
Also, if $\psi \in \textup{Hom}^{D}\big(W, V|_{D}\big)$ then we have that the following diagram commutes:
\begin{equation}\tag{$\dagger\dagger$}\label{commute2}
    \begin{tikzcd}
{\Hom_k\big(D,W\big)} \arrow[dd, "{\Hom_k\big(D,\psi\big)}"] \arrow[rrr, "{\theta_{W}}"] &  &  & {W} \arrow[dd, "\psi"] \\
     & & &      \\
    {\Hom_k\big(D,V\big)} \arrow[rrr, "\theta_{V}"] &  & & V.
    \end{tikzcd} 
\end{equation}
We want to show that if $\phi \in \Hom^{C}\left(\textup{Ind}_{D}^{C}(W), V\right)$, then $\Gamma(\phi) \in \textup{Hom}^{D}\big(W, V|_{D}\big)$. This amounts to showing that diagram \eqref{commute2} commutes when setting $\psi = \Gamma(\phi).$ Let $\alpha \in \Hom_k(D,W).$ Travelling horizontally and then vertically along (\ref{commute2}) gives 
\[ \Big(\Gamma(\phi) \circ \theta_W\Big) (\alpha) = \phi\Big(c \longmapsto \epsilon(c)\theta_W(\alpha) \Big). \]
On the other hand, travelling vertically and then horizontally gives 
\[ \theta_V\big(\Gamma(\phi)\circ \alpha\big) = \theta_V\Big(d \longmapsto \phi\big(c \longmapsto \epsilon(c)\alpha(d)\big)\Big).\]
We must show that these elements of $V$ are equal. Observe that since $\phi \circ g = \phi \circ f$ (by the identification (\ref{ident})) we have as maps from $\Hom_k(D \otimes C,W) \cong \Hom_k\big(C, \Hom_k(D,W)\big)$ to $V$, 
\begin{multline*}
\phi \circ \Hom_k\big(C,\theta_W)\, \stackrel{\text{def}}{=}\, \phi \circ g\, =\, \phi \circ f\, \\\stackrel{\text{def}}{=} \,\left(\phi \circ \theta_{\Hom_k(C,W)}\right)|\hspace{-0.45ex}_{|\Hom_k\big(C,\Hom_k(D,W)\big)} \,\stackrel{\eqref{commute1}}{=}\, \theta_V \circ \Hom_k\big(C,\phi\big)|\hspace{-0.45ex}_{|\Hom_k\big(C,\Hom_k(D,W)\big)}.    
\end{multline*}
Applying this to $\theta_V\big(\Gamma(\phi)\circ \alpha\big)$ (specifically, at the equality marked by $*$ in what follows) gives:
\begin{align*}
    \theta_V\big(\Gamma(\phi)\circ \alpha\big) &= \theta_V\Big(d \mapsto \phi\big(c \mapsto \epsilon(c)\alpha(d)\big)\Big)\stackrel{*}{=} \phi\Big(c \longmapsto \theta_W\big(\gamma \mapsto \epsilon(c)\alpha(\gamma)\big)\Big) \\
    &= \phi\Big(c \longmapsto \epsilon(c)\theta_W\big(\gamma \mapsto \alpha(\gamma)\big)\Big)
    = \Big(\Gamma(\phi) \circ \theta_W\Big) (\alpha)
\end{align*}
as required. Injectivity is obvious. Surjectivity amounts to checking that if $\psi \in \Hom^{D}\big(W,V|_{D}\big)$ then $\Gamma^{-1}(\psi) \equiv \theta_V \circ \Hom_k(C,\psi)$ satisfies $\Gamma^{-1}(\psi) \circ f \equiv \Gamma^{-1}(\psi) \circ g$, i.e., that the following diagram commutes:

\begin{equation*}
    \begin{tikzcd}
{\Hom_k\Big(C,\Hom_k\big(D,W\big)\Big)} \arrow[rrr, "{f}"] \arrow[dd, "{g}"] &  & &{\Hom_k\big(C,W\big)} \arrow[rr, "{\Hom_k(C,\psi)}"] && {\Hom_k(C,V)} \arrow[dd,"{\theta_V}"] \\
& & & & &   \\
{\Hom_k\big(C,W\big)} \arrow[rrr, "{\Hom_k(C,\psi)}"] & & & {\Hom_k(C,V)} \arrow[rr, "\theta_V"]& & {V}.
    \end{tikzcd}
\end{equation*}
This follows immediately once we observe that 
\[\Hom_k(C,\psi) \circ g \equiv  \Hom_k\big(C,\theta_V|_{D}\big) \circ \Hom_k\big(C,\Hom_k(D,\psi)\big)\] (this is just the functor $\Hom_k(C,-)$ applied to the definition of $\psi$ being a $D$-contramodule homomorphism), and noticing that the resulting diagram

\begin{equation*}
    \begin{tikzcd}
{\Hom_k\Big(C,\Hom_k\big(D,W\big)\Big)} \arrow[rrrrr, "{\theta_{\Hom_k(C,W)|_{\Hom_k\left(C,\Hom_k(D,W)\right)}}}"] \arrow[dd, "{\Hom_k\big(C, \Hom_k(D,\psi)\big)}"] &  & & & &{\Hom_k\big(C,W\big)} \arrow[rr, "{\Hom_k(C,\psi)}"] && {\Hom_k(C,V)} \arrow[dd,"{\theta_V}"] \\
& & & & & & &  \\
{\Hom_k\big(C,\Hom_k(D,V)\big)} \arrow[rrrrr, "{\Hom_k(C,\theta_V)}"] & & & & & {\Hom_k(C,V)} \arrow[rr, "\theta_V"]& & {V}
    \end{tikzcd} \end{equation*}
commutes since it is nothing more than the condition that $\Gamma^{-1}(\psi) \in \Hom^{C}\big(\Hom_k(C,W),V\big)$ after restricting ${\Hom_k\big(C,\Hom_k(C,W)\big)}$ to ${\Hom_k\big(C,\Hom_k(D,W)\big)}$.
\end{proof}

\section{Affine quotients and contra-exactness.}\label{Section-affine}

Let $G$ be an algebraic group over an algebraically closed field $k$. Let $H$ be a closed subgroup of $G$. The subgroup $H$ is exact in $G$ if the induction functor of modules is exact. We may also define the notion of contra-exactness, namely, we say that $H$ is \textit{contra-exact} in $G$ if the induction functor from $k[H]$-Contra to $k[G]$-Contra (defined using the map $\iota^*: k[G] \rightarrow k[H]$ induced from the embedding $\iota: H \rightarrow G)$ is exact. That is, if the functor takes short exact sequences of $k[H]$-contramodules to short exact sequences of $k[G]$-contramodules. 

In this section, we collect results from other sources and show an analogous result for contramodules to that of Cline, Parshall and Scott \cite{cline1977induced}. To begin, the following result relates the affineness of the quotient $G/H$ to the injectivity of $k[G]$ as a $k[H]$-comodule.

\begin{proposition} \cite{cline1977induced}\label{cline_affine_iff_injective}
Let $H$ be a closed subgroup of an algebraic group $G$ with coordinate rings $k[H]$ and $k[G]$, respectively. Then the following are equivalent:
\begin{enumerate}[label=\roman*)]
    \item the quotient $G/H$ is an affine variety
    \item $k[G]$ is an injective $k[H]$-comodule
    \item $H$ is exact in $G$.
\end{enumerate}
\end{proposition}

In the next result, we utilise opposite categories to ensure that all functors are covariant. Recall that given a category $C$ the opposite category $C^{\text{op}}$ is the category with the same objects as $C$ and morphisms given by $\Hom_{C^{\text{op}}}(X,Y) := \Hom_{C}(Y,X)$. The following result gives us a relation between the injectiveness of a comodule and the exactness of the Cohom functor. It also relates the projectiveness of a contramodule to the exactness of the Cohom functor, which we state now to be used in the next section. 

\begin{lemma}\cite[Section 3.1]{positselski2021contramodules}\label{Posit_equivalences}
Let $C$ be a coassociative coalgebra over a field $k$, then 
\begin{enumerate}[label=\alph*)]
    \item if $M$ is a $C$-comodule, the following are equivalent:
    \begin{enumerate}[label=\roman*)]
        \item $M$ is an injective $C$-comodule
        \item The functor $\text{Cohom}_C(M,-):\,C$-Contra $\longrightarrow Vect$ is exact
        \item The functor $M \boxtimes_C -:\,$$C$-Comod$_\text{left}  \longrightarrow Vect$ is exact.
    \end{enumerate}
    \item if $B$ is a $C$-contramodule, the following are equivalent:
    \begin{enumerate}[label=\roman*)]
        \item $B$ is a projective $C$-contramodule
        \item The functor $\text{Cohom}_C(-,B):\,C$-$\text{Comod}^{\text{op}}$ $\longrightarrow Vect$ is exact
        \item The functor $B \odot_C -:\,C$-Comod$_\text{left} \longrightarrow Vect$ is exact.
    \end{enumerate}
\end{enumerate}

\end{lemma}

Putting these two results together, we obtain our final result for this section.

\begin{theorem}
Let $H$ be a closed subgroup of an algebraic group $G$ with coordinate rings $k[H]$ and $k[G]$, respectively. Then the following are equivalent:
\begin{enumerate}[label=\roman*)]
    \item $G/H$ is an affine variety
    \item $H$ is contra-exact in $G$
    \item $H$ is exact in $G$.
\end{enumerate}
\end{theorem}

\begin{proof}
Proposition \ref{cline_affine_iff_injective} gives $i) \iff iii)$. It suffices to show that $i) \iff ii)$. By combining Proposition \ref{cline_affine_iff_injective} and Lemma \ref{Posit_equivalences} we have that the following are equivalent:
\begin{itemize}
    \item[--] $G/H$ is an affine variety
    \item[--] $k[G]$ is an injective $k[H]$-comodule
    \item[--] the functor $\textup{Cohom}_{k[H]}(k[G],-):\,k[H]$-Contra $\longrightarrow \text{Vect}$ is exact.
\end{itemize}
Now, we have $\textup{Cohom}_{k[H]}(k[G],-) = F \circ \textup{Ind}_{k[H]}^{k[G]}(-)$, where $F: k[G]$-Contra $\longrightarrow $ Vect is the forgetful functor. Since the forgetful functor is faithful it reflects monomorphisms. Thus, $\textup{Cohom}_{k[H]}(k[G],-)$ being exact is equivalent to $\textup{Ind}_{k[H]}^{k[G]}(-)$ being exact. This concludes the proof.
\end{proof}

\section{An inverse limit theorem for projective contramodules of semi\-simple algebraic groups.}\label{Section-inverse-limit}

In this section, we prove an analogous result to the direct limit theorem \cite{donkin1980}, \cite[Proposition 2]{Ballard1978}. Vaguely speaking, the direct limit theorem says that given an algebraic group $G$ one can produce injective covers of the simple $G$-modules by taking a direct limit of certain twisted tensor products of $G$-structures of injective covers of simple $G_1$-modules. We give two proofs of our contramodule analog. The first proof is direct, it follows a strategy similar to Donkin's proof of the direct limit theorem whilst making use of contramodule constructions. The second proof is based on a conjecture given to me in a correspondence with Donkin. This proof is more conceptual. We include both proofs as it is believed they each hold value.

Let $G$ be a simply connected, semisimple algebraic group over an algebraically closed field $k$ of characteristic $p$. Let $k[G]$ denote the coordinate ring of $G$, and let $F: G \longrightarrow G^{(1)}$ be the Frobenius map given by $F^*: k[G]^{(1)} \longrightarrow k[G]$, $F^*(f) = f^p$ for $f \in k[G]$. Although confusion is unlikely, we explicitly note that the use of $F$ for the forgetful functor (as used in the previous section) has been dropped and in this section, $F$ will always denote the Frobenius map.  For each $r > 0$, let $G_r$ denote the $r^{\text{th}}$ Frobenius kernel of $G$, and let $k[G_r]$ be the coordinate ring of $G_r$. We identify $u_r$, the $r^{th}$ hyperalgebra of $G$, with $k[G_r]^*$. 

Let $T \subset G$ be a maximal torus, and let $X(T)$ $\big($resp. $X_+(T)\big)$ denote the set of weights (resp. dominant weights). For each dominant weight $\lambda \in X_+(T)$, let $L(\lambda)$ denote the simple $G$ module of highest weight $\lambda$. For each $p^r$-restricted weight $\lambda$, let $L_r(\lambda)$ denote the simple $G_r$-module with highest weight $\lambda$. Note that as $G_r$-modules we have $L(\lambda)|_{G_r} \cong L_r(\lambda)$. Such modules describe all simple $G_r$ modules up to isomorphism \cite[Chapter II.3]{jantzen2003representations}. Let $V$ be a $G_1$-module, then we say that a $G$-module $W$ is a $G$-structure on $V$ if $W|_{G_1} \cong V$ as $G_1$-modules.

\begin{Assumption*}
We assume that $G$-structures $P(\lambda)$ exist on $P_1(\lambda)$, the $G_1$-projective covers of $L_1(\lambda)$, for each $p$-restricted weight $\lambda$. 

\end{Assumption*}

\begin{remark}\label{HumVermaconjecture}
It is a long-standing conjecture of Humphreys and Verma that such $G$-structures exist, and they have been shown to exist for $p \geq 2h - 4$, where $h$ is the Coxeter number of $G$ \cite[Chapter II.11]{jantzen2003representations} \cite{bendel2023donkins}.
\end{remark}

Let $\lambda$ be a dominant weight with $p$-adic expansion 
\[ \lambda = \sum_{i=0}^sp^i\lambda_i\]
with each $\lambda_i$ a $p$-restricted weight. Suppose that for each $i$, $P(\lambda_i)$ is a $G$-structure on $P_1(\lambda_i)$, the $G_1$-projective cover of $L_1(\lambda_i)$. For any $r > s$ define
\[P_{\lambda,r} = P(\lambda_0) \otimes P(\lambda_1)^{(1)} \otimes \dots \otimes P(\lambda_s)^{(s)} \otimes P(0)^{(s+1)} \otimes \dots \otimes P(0)^{(r-1)}\]
where the superscript $(i)$ denotes a twist by $i$ applications of the Frobenius morphism.

As $G_r$-modules, we have $P_{\lambda,r} \cong P_r(\lambda)$, and so in particular $P_{\lambda,r} \big/\text{ Rad}_{G_r}(P_{\lambda,r}) \cong L_r(\lambda)$. This follows from the analogous discussion for the injective construction $Q_r(\lambda)$ (denoted $Q(r,\lambda)$ in Donkin \cite{donkin1980}), along with the fact that the Frobenius kernels are unimodular (i.e. $P_r(\lambda) = Q_r(\lambda)$ for each $r > 0$ and $\lambda$ a $p^r$-restricted weight \cite[Chapter 6, Theorem 6]{alperin_1986}).

Now, let $\lambda$ be a fixed dominant weight, then for any finite-dimensional $k[G]$-comodule $V$, we let $f(V)$ denote the composition multiplicity of $L(\lambda)$ in $V$. We have the following lemma: 

\begin{lemma}\label{DonkinLemmaAnalog}
Let $V$ be a finite-dimensional $k[G]$-comodule, then for all large $r$ (depending on $V$) we have 
\[ \Hom_{k[G]}\big(P_{\lambda,r},V\big) = \Hom_{k[G_r]}\big(P_{\lambda,r},V\big)\]
and has dimension $f(V)$.
\end{lemma}

\begin{proof}
We refer to the analogous result and proof by Donkin \cite{donkin1980}. The main points are that the proof uses induction on the composition length of $V$, and the largeness of $r$ depends on the weights of $V$. \cite[Section 2.3]{HUMPHREYS1977203}. The proof of our result follows almost verbatim, so we will not reproduce it.
\end{proof}

With Lemma \ref{Posit_equivalences}b) in mind, our plan of attack in showing that our (not yet stated) contramodule of interest is projective will be to show that the associated Cohom functor is exact. This idea leads us to our next result. Firstly however, we recall a couple of constructions and results from Section \ref{Section-contra-defns}, as well as describe a correspondence between finite-dimensional $C$-comodules and $C$-contramodules.

Let $V$ and $W$ be finite-dimensional $C$-comodules\footnote{Of course, we only need at least one of $V$ and $W$ to be finite-dimensional for this to hold, however, we make use of this description only in cases where $V$ and $W$ are both finite-dimensional, so we may as well assume it from the outset.}, then we have:

\begin{align*}
    \Hom_C(W,V) \cong & \,\,V \boxtimes_C W^* \\
    =& \begin{tikzcd}[scale cd=1, sep = huge, ampersand replacement=\&]\text{eq}\,\bigg(V \otimes_k W^* \ar[r,shift left=.75ex,"\text{Id}_V \otimes \Delta_{W^*}"]
  \ar[r,shift right=.75ex,swap,"\Delta_V \otimes \text{Id}_{W^*}"] \& V \otimes_k C \otimes_k W^* \bigg).
 \end{tikzcd}
\end{align*}
Furthermore, as $W$ is a finite-dimensional $C$-comodule, we may view $W$ as a (right) $C$-contramodule via the following composition:
\[ \theta_W : \Hom_k(C,W) \cong C^* \otimes W \xrightarrow{\text{Id}_{C^*} \otimes \Delta_W} C^* \otimes W \otimes C \xrightarrow{\text{eval}} W\]
where $\text{eval}:C^* \otimes W \otimes C \longrightarrow W$ is the obvious evaluation map involving $C^*$ and $C$. A key observation we will make use of is that this $C$-contramodule structure is nothing more than the natural (left) $C^*$-module structure on a $C$-comodule. Thus, if $m_W$ denotes the $C^*$-action on $W$ as a module, then $\theta_W \equiv m_W$ (once we identify $\Hom_k(C,W)$ with $C^* \otimes W$).

By viewing $W$ as a contramodule in this way, we may consider the cohomomorphisms from the comodule $V$ to the contramodule $W$. We have:

\begin{align*}
    \text{Cohom}_C(V,W) \cong & \,\,W \odot_C V^* \\
    =& \begin{tikzcd}[scale cd=1, sep = huge, ampersand replacement=\&]\text{coeq}\,\bigg(\Hom_k(C,W) \otimes_k V^*  \ar[rrr,shift left=0.9ex,"\big(\text{eval} \otimes \text{Id}_{V^*}\big)\circ\big(\text{Id}_{\Hom_k(C,W)} \otimes \Delta_{V^*} \big)"]
  \ar[rrr,shift right=.9ex,swap,"\theta_W \otimes\, \text{Id}_{V^*}"] \&\&\& W \otimes_k V^*\bigg)
 \end{tikzcd}\\
 =&  \begin{tikzcd}[scale cd=1, sep = huge, ampersand replacement=\&]\text{coeq}\,\bigg(V^* \otimes_k C^* \otimes_k W \ar[r,shift left=0.9ex,"m_{V^*} \otimes \text{Id}_W "]
  \ar[r,shift right=.9ex,swap,"\text{Id}_{V^*} \otimes m_{W}"] \& V^* \otimes_k W\bigg).
 \end{tikzcd}\\
\end{align*}
Having recalled these constructions, we give the following result. 

\begin{lemma}[$\Hom - \text{Cohom} $ duality.]\label{duality_lemma} Let $C$ be a coalgebra over a field $k$, and let $(V,\Delta_V)$ and $(W,\Delta_W)$ be finite-dimensional  $C$-comodules, then 
\[\textup{Cohom}_C(V,W) = \Hom_C(W,V)^*,  \]
where on the left-hand side $W$ is viewed as a $C$-contramodule in the way described above. 

In particular, by considering the cases when $C=k[G]$, $W = P_{\lambda,r}$, and also $C=k[G_r]$, $W = P_{\lambda,r}$ we obtain that for all finite-dimensional $k[G]$-comodules $V$
\[\textup{Cohom}_{k[G]}(V, P_{\lambda,r}) = \textup{Cohom}_{k[G_r]}(V, P_{\lambda,r}) \]
for all large $r$ and has dimension $f(V)$.
\end{lemma}

\begin{proof}
Taking the description of $\Hom_C(W,V) \cong V \boxtimes_C W^*$ and dualising it gives
\begin{align*}
\Hom_C(W,V)^* &= \begin{tikzcd}[scale cd=1, sep = huge, ampersand replacement=\&]\text{coeq}\,\bigg(\Big(V \otimes_k C \otimes_k W^*\Big)^* \ar[r,shift left=0.9ex,"(\text{Id}_V \otimes \Delta_{W^*})^*"]
  \ar[r,shift right=.9ex,swap,"(\Delta_{V} \otimes \text{Id}_{W^*})^*"] \& \Big(V \otimes_k W^*\Big)^*\bigg)
 \end{tikzcd} \\
 &= \begin{tikzcd}[scale cd=1, sep = huge, ampersand replacement=\&]\text{coeq}\,\bigg(V^*\otimes_k C^* \otimes_k W \ar[r,shift left=0.9ex,"\text{Id}_{V^*} \otimes (\Delta_{W^*})^*"]
  \ar[r,shift right=.9ex,swap,"(\Delta_{V})^* \otimes \text{Id}_{W}"] \& V^* \otimes_k W\bigg),
 \end{tikzcd} \\
 \end{align*}
 where we have identified $\big(V \otimes C \otimes W^*)^*$ with $V^* \otimes C^* \otimes W$. (Note that this is only valid because both $V$ and $W$ are finite-dimensional.) Comparing this to the description of $\text{Cohom}_C(V,W) \cong W \odot_C V^*$, all that remains is to observe that $(\Delta_{V})^* \equiv m_{V^*}$ and $(\Delta_{W^*})^* \equiv m_W$. The second part of the lemma follows from the first part, along with Lemma \ref{DonkinLemmaAnalog}.
\end{proof}
We now fix a projection $q: P(0) \onto L(0)$ (using Lemma \ref{f.d.-comod-to-contra} we may view this as a map of contramodules), and let $P_\lambda = \varprojlim P_{\lambda,r}$ be the inverse limit defined in $k[G]$-Contra, where the transition maps $
q_r: P_{\lambda,r} \longrightarrow P_{\lambda, {r-1}}$ are given by 
\[q_r: P_{\lambda, r} = P_{\lambda,r-1} \otimes P(0)^{(r-1)}\, \xrightarrow{\text{Id}_{P_{\lambda,r-1}} \otimes \,q^{(r-1)}}\, P_{\lambda,r-1} \otimes L(0)^{(r-1)}\, \cong\, P_{\lambda,r-1}.\]
Note that the maps $q_r$ are contramodule homomorphisms since $q$ is. We have natural projection maps $\pi_r: P_{\lambda} \longrightarrow P_{\lambda,r}$. We aim to show that $P_\lambda$ is the projective cover of $L(\lambda)$ in $k[G]$-Contra. By Lemma \ref{Posit_equivalences} we know that $P_{\lambda}$ is a projective contramodule if and only if the functor \[\text{Cohom}_{k[G]}( -,P_{\lambda})\,:\,\, k[G]\text{-Comod}^{\text{op}}\longrightarrow \text{Vect}\] is exact, and this is ultimately what we aim to show. First, we state and prove two lemmas which will aid us in showing the exactness of the aforementioned functor. 

\begin{lemma}\label{local_finiteness}
$P_\lambda$ is a projective contramodule if and only if \[\text{Cohom}_{k[G]}( -,P_{\lambda})\,:\,\,\, k[G]\text{-Comod}^{\text{op}}_{\text{f.d.}} \longrightarrow \text{Vect}\] is exact, that is, we only need to show exactness for finite-dimensional $k[G]$-comodules.
\end{lemma}

\begin{proof}
    By Lemma \ref{Posit_equivalences}, we have that the following are equivalent:
    \begin{itemize}
        \item[--] $P_\lambda$ is a projective contramodule
        \item[--] $\text{Cohom}_{k[G]}( -,P_\lambda)$ is exact
        \item[--] $P_\lambda \odot_{k[G]} -$ is exact.
    \end{itemize}
    However, every comodule is a direct limit of its finite-dimensional subcomodules, and the contratensor product commutes with colimits in the comodule argument \cite[Proof of Lemma 3.1]{positselski2021contramodules}. Thus, $P_\lambda \odot_{k[G]} -$ is exact if and only if $P_\lambda \odot_{k[G]} -$ is exact for finite-dimensional comodules.
    
    Finally, for a finite-dimensional $k[G]$-comodule, $V$, we have $P_{\lambda} \odot_{k[G]} V\,\cong\, \text{Cohom}_{k[G]}( V^*,P_{\lambda})$, via the natural identification $(V^*)^*\,\cong\,V$. So we see that $P_{\lambda}$ is projective if and only if $\text{Cohom}_{k[G]}( -,P_{\lambda})$ is exact for finite-dimensional comodules.
\end{proof}
Our next result will require some results regarding Mittag-Leffler systems, so we briefly review them now. (For more details on this see \cite[Chapter 11.3]{geoghegan2007topological}.) Let $R$ be a ring, and let $I$ be a directed set. Let $\{A_i, f_{ij}: A_j \rightarrow A_i$ for $j \geq i\}$ be an inverse system of $R$-modules (we will only need this for $R=k$). Observe that for each fixed $i$, the subsets $f_{ij}(A_j) \subset A_i,$ $j>i$ form a decreasing family of submodules. We say the system is $\textit{Mittag-Leffler}$ if this decreasing family of submodules eventually stabilises. That is, for each fixed $i$, there exists a $j>i$ such that for all $k>j$ we have $f_{ik}(A_k) = f_{ij}(A_j)$. We have the following two examples of systems which are Mittag-Leffler, both of which will be used later. 

\begin{example}
Let $I$ be a directed set, and $\{A_i, f_{ij}:A_j \rightarrow A_i, j \geq i\}$ be an inverse system of $R$-modules, then the system is Mittag-Leffler if
\begin{itemize}
    \item The $f_{ij}$ are surjective for all $i,j \in I$ where $j > i$.
    \item The $f_{ij}$ are surjective eventually, that is, there exists a $N \in I$ such that $f_{n,j}$ is surjective for all $j > n > N$.
\end{itemize} 
\end{example}
Such systems will be very useful for us in light of the following lemma and corollary.

\begin{lemma}\cite[Theorem 11.3.2]{geoghegan2007topological}
Let $I$ be a countable directed set and $\{A_i\}$, $\{B_i\}$ and $\{C_i\}$ be inverse systems of $R$-modules indexed by $I$ such that 
\[0 \longrightarrow A_i \longrightarrow B_i \longrightarrow C_i \longrightarrow 0\]
is exact for all $i$. Then if $\{A_i\}$ is Mittag-Leffler we have that
\[0 \longrightarrow \varprojlim A_i \longrightarrow \varprojlim B_i \longrightarrow \varprojlim C_i \longrightarrow 0\]
is exact.
\end{lemma}

\begin{corollary}\label{MLcorollary}
Let $I$ be a countable directed set and $\{A_i\}$, $\{B_i\}$, $\{C_i\}$ and $\{D_i\}$ be inverse systems of $R$-modules indexed by $I$ such that $\displaystyle 0 \longrightarrow A_i \longrightarrow B_i \longrightarrow C_i \longrightarrow D_i \longrightarrow 0$ is exact for all $i$. Suppose further that
\begin{enumerate}[label=\roman*)]
    \item $\{A_i\}$ is Mittag-Leffler,
    \item $\left\{B_i \,\big/\, \textup{Im}(A_i \hookrightarrow B_i)\right\}$ is Mittag-Leffler.
\end{enumerate}
Then $\displaystyle 0 \longrightarrow \varprojlim A_i \longrightarrow \varprojlim B_i \longrightarrow \varprojlim C_i \longrightarrow \varprojlim D_i  \longrightarrow 0$ is exact.
\end{corollary}

\begin{proof}
For each $i$, we may form the three term exact sequences
\[ \displaystyle 0 \longrightarrow B_i \,\big/\, \textup{Im}(A_i \hookrightarrow B_i) \longrightarrow C_i \longrightarrow D_i \longrightarrow 0.\]
Since $\left\{B_i \,\big/\, \textup{Im}(A_i \hookrightarrow B_i)\right\}$ is Mittag-Leffler by assumption, we have by the lemma that 
\[\displaystyle 0 \longrightarrow \varprojlim\Big( B_i \,\big/\, \textup{Im}(A_i \hookrightarrow B_i)\big) \longrightarrow \varprojlim C_i \longrightarrow \varprojlim D_i  \longrightarrow 0\] is exact. We now require that
\[ \varprojlim\Big( B_i \,\big/\, \textup{Im}(A_i \hookrightarrow B_i)\Big) = \Big( \varprojlim B_i\Big) \,\big/\, \Big(\textup{Im}\big(\varprojlim A_i \hookrightarrow \varprojlim B_i\big)\Big),\]
but this follows (thanks to the lemma) from taking the inverse limit of the family of exact sequences
\[ \displaystyle 0 \longrightarrow A_i \longrightarrow B_i \longrightarrow B_i \,\big/\, \textup{Im}(A_i \hookrightarrow B_i) \longrightarrow 0,\]
along with the fact that $\{A_i\}$ is Mittag-Leffler. \end{proof}

We may now state the second of our two lemmas which will be used in proving Theorem \ref{P_lambda_main_thm}. 

\begin{lemma}\label{lim_commutes_with_cohom} Let $G$ be a simply connected, semisimple algebraic group over an algebraically closed field $k$ of characteristic $p$ and let $\lambda$ be a dominant weight. Let $V$ be a finite-dimensional $k[G]$-comodule. Then we have 
\[\textup{Cohom}_{k[G]}(V,P_{\lambda})\, =\, \varprojlim \textup{Cohom}_{k[G]}(V,P_{\lambda,r}).\]
\end{lemma}

\begin{proof}

Recall that
$\begin{tikzcd}[scale cd=1]\text{Cohom}_{k[G]}(V,P_{\lambda,r}) = \text{coeq}\bigg(\Hom_k(V \otimes k[G],P_{\lambda,r}) \ar[r,shift left=.75ex,"f_r"]
  \ar[r,shift right=.75ex,swap,"g_r"]
&
\Hom_k(V,P_{\lambda,r})\bigg),\end{tikzcd}$
where
   \[ f_r \equiv \Hom_k(\Delta_V,P_{\lambda,r}): \Hom_k(V \otimes k[G], P_{\lambda,r}) \longrightarrow \Hom_k(V,P_{\lambda,r})  \]
is the map induced from the coaction on the $k[G]$-comodule $V$, and
    \[g_r \equiv \Hom_k(V, \theta_{P_{\lambda,r}}) : \Hom_k\big(V, \Hom_k(k[G],P_{\lambda,r})\big) \longrightarrow \Hom_k(V,P_{\lambda,r})\]
is the map induced by the contra-action on the $k[G]$-contramodule $P_{\lambda,r}$, where we have used the tensor-hom adjunction to identify $\Hom_k\big(V, \Hom_k(k[G],P_{\lambda,r})\big)$ and $\Hom_k(V \otimes k[G], P_{\lambda,r})$.
Thus, for each $r$, we have the following four-term exact sequence: 
\[0 \rightarrow \ker(f_r - g_r) \longrightarrow \Hom_k(V \otimes k[G],P_{\lambda,r}) \xrightarrow{f_r -g_r} \Hom_k(V,P_{\lambda,r}) \longrightarrow \text{Cohom}_{k[G]}(V,P_{\lambda,r}) \rightarrow 0.\]
Suppose for now that this system of exact sequences is preserved under taking the inverse limit, i.e., the following sequence is exact:
\[0 \longrightarrow \ker(f - g) \longrightarrow \Hom_k(V \otimes k[G],P_{\lambda}) \xrightarrow{f -g} \Hom_k(V,P_{\lambda}) \longrightarrow \varprojlim \text{Cohom}_{k[G]}(V,P_{\lambda,r}) \longrightarrow 0\]
where $f := \varprojlim f_r$, $g := \varprojlim g_r$ and we have used that limits commute with the covariant Hom functor. Then we obtain that
\[\varprojlim \text{Cohom}_{k[G]}(V,P_{\lambda,r}) = \textup{coker}\Big(\Hom_k(V \otimes k[G],P_{\lambda}) \xrightarrow{f -g} \Hom_k(V,P_{\lambda}) \Big) = \text{Cohom}_{k[G]}(V,P_{\lambda})\]
and the lemma is proved. Thus, it remains to show that our system of exact sequences is preserved under inverse limits. By Corollary \ref{MLcorollary}, it suffices to show that the following hold:
\begin{enumerate}[label = \roman*)]
    \item $\big\{\ker(f_r - g_r)\big\}$ is Mittag-Leffler
    \item $\left\{\Hom_k(V \otimes k[G],P_{\lambda,r})\Big/\ker(f_r,g_r)\right\}$ is Mittag-Leffler
\end{enumerate}
where in both cases the transition maps are the obvious ones. 

We begin with $ii)$. Let $q_r:P_{\lambda,r} \onto P_{\lambda,r-1}$ denote the transition map from $P_{\lambda,r}$ to $P_{\lambda,r-1}$. Then we have the following diagram, which we claim commutes:

\[\begin{tikzcd}
\Hom_k(V \otimes k[G], P_{\lambda,r}) \arrow[dd, "q_{r_*}"] \arrow[rrr, "f_r - g_r"] &  &  & \Hom_k(V, P_{\lambda,r}) \arrow[dd, "q_{r_*}"] \\
     & & &      \\
    \Hom_k(V \otimes k[G], P_{\lambda,r-1}) \arrow[rrr, "f_{r-1} - g_{r-1}"] &  & & \Hom_k(V, P_{\lambda,r-1}) \\
\end{tikzcd}\]
where by abuse of notation, $q_{r_*}$ denotes both of the induced maps appearing on the vertical arrows in the diagram. Indeed, let $\varphi \in \Hom_k(V \otimes k[G], P_{\lambda,r}),$ then we have
\begin{align*}
q_{r_*} \circ (f_r - g_r)\big(\varphi\big) &= q_r \circ \varphi \circ \Delta_V - q_r \circ \theta_{P_{\lambda,r}} \circ \varphi \\
&= q_r \circ \varphi \circ \Delta_V -  \theta_{P_{\lambda,r-1}} \circ q_r \circ \varphi \\
&= (f_{r-1}-g_{r-1})\circ q_{r_*}\big(\varphi\big)
\end{align*}
where $q_r \circ \theta_{P_{\lambda,r}} \equiv \theta_{P_{\lambda,r-1}} \circ q_r$ precisely because $q_r$ is a contramodule homomorphism. Thus, we have that $q_{r_*}\big(\ker(f_r-g_r)\big) \subset \ker(f_{r-1}-g_{r-1})$. This gives us a factoring:
\[\begin{tikzcd}
\Hom_k(V \otimes k[G], P_{\lambda,r}) \arrow[dr, two heads] \arrow[r, two heads, "q_{r_*}"] & \Hom_k(V \otimes k[G], P_{\lambda,r-1}) \arrow[r, two heads] & \dfrac{\Hom_k(V \otimes k[G], P_{\lambda,r-1})}{\ker(f_{r-1} -g_{r-1})}. \\
& \dfrac{\Hom_k(V \otimes k[G], P_{\lambda,r})}{\ker(f_{r} -g_{r})} \arrow[ur, two heads] & 
\end{tikzcd}\]
Therefore, $\left\{\dfrac{\Hom_k(V \otimes k[G], P_{\lambda,r})}{\ker(f_{r} -g_{r})}\right\}$ is Mittag-Leffler, since the transition maps are surjective.

We are left to show $i)$, i.e., that $\{\ker(f_r-g_r)\}$ is Mittag-Leffler. To do this we will show that the transition maps $\left\{q_{r_*}|_{\ker(f_r-g_r)}\right\}$ are surjective eventually, that is, surjective for large enough $r$. To begin, consider the following diagram (note that all squares commute):

\[\begin{tikzcd}[sep=small]
0 \arrow[r] & \ker(f_r-g_r) \arrow[ddd] \arrow[r, hook] & \big[V \otimes k[G], P_{\lambda,r}\big] \arrow[ddd, two heads, "q_{r_*}"] \arrow[rr, "f_r - g_r"] && \big[V, P_{\lambda,r}\big] \arrow[r] \arrow[ddd, two heads, "q_{r_*}"] & \text{Cohom}_{k[G]}(V,P_{\lambda,r}) \arrow[r] & 0 \\
&&&&&&\\
&&&&&&\\
0 \arrow[r] & \ker(f_{r-1}-g_{r-1}) \arrow[r, hook] & \big[V \otimes k[G], P_{\lambda,r-1}\big] \arrow[rr, "f_{r-1} - g_{r-1}"] & \hspace{2ex}& \big[V, P_{\lambda,r-1}\big] \arrow[r] & \text{Cohom}_{k[G]}(V,P_{\lambda,r-1}) \arrow[r] & 0,
\end{tikzcd}\]
where $[-,-]$ is used to denote $\Hom_k(-,-)$ for typographical reasons. Let $\varphi \in \ker(f_{r-1}-g_{r-1})$. As $q_{r_*} \equiv \Hom_k(V \otimes k[G], q_r)$ is surjective there exists $\widetilde{\varphi} \in \Hom_k(V \otimes k[G],P_{\lambda,r})$ such that $q_{r_*}(\widetilde{\varphi}) \equiv \varphi$. Now, by the commutative diagram we must have \[(f_r-g_r)(\widetilde{\varphi}) \in \ker\big(q_{r_*}: \Hom_k(V,P_{\lambda,r}) \longrightarrow \Hom_k(V,P_{\lambda,r-1})\big),\] where $\ker(q_{r_*})$ consists of all linear maps $f:V \longrightarrow P_{\lambda,r}$ such that $q_r \circ f \equiv 0$, i.e. all maps $f$ such that $\text{Im}(f) \subset \ker(q_r) = P_{\lambda,r-1} \otimes \ker\big(q^{(r-1)}\big)$, where$\begin{tikzcd}[sep = small]q: P(0) \arrow[r,two heads] & L(0)\end{tikzcd}$ is our fixed projection. 

We now construct a map $\psi \in \Hom_k\big(V \otimes k[T], \ker({q_r})\big) \subset \Hom_k\big(V \otimes k[G], \ker({q_r})\big)$ with the same image under $f_r-g_r$ as $\widetilde{\varphi}$. To do so we make use of the fact that we may take $r$ to be large. Precisely, we can take $r$ large enough so that all weights of $V$ are $p^{r-1}$-restricted. Let $l = $rank$(G)$, then $k[T] \cong k[x_1,...,x_l]$. For any (finite-dimensional) $k[T]$-comodule W, let $W_\lambda = \{w \in W : \Delta_W(w) = w \otimes \lambda \}$ denote the $\lambda$-weight space. Then we have that $\dim(P(0)_{1}) = 1$, in particular, $\ker(q)_1 = \{0\}$. (See \cite[pp. 148-149]{Donkin_note_on_charcters} for the injective result, the proof in the projective case is similar.) Thus, the weights appearing in the weight space decomposition of $\ker\big(q^{(r-1)}\big)$ are all of the form $x_1^{k_1}...x_l^{k_l}$ with at least one $k_i \geq p^{r-1}.$ 

Now, fix a basis $\{v_i\}$ of $V$, and also fix a basis $\{p_j\}$ of $\ker(q_r).$ We may as well assume that the chosen bases are weight bases. Thus, let $v_i$ have weight $\lambda_i$ and $p_j$ have weight $\mu_j$  for each $i$ and $j$, where $i$, $j$ span over the indices of the basis of $V$ and $\ker(q_r),$ respectively. Now, explicitly describing $(f_r-g_r)(\widetilde{\varphi})$ on the basis $\{v_i\}$ of $V$ we have: 
\[\Big((f_r-g_r)(\widetilde{\varphi})\Big)(v_i) = \sum_{j} a_{ij}p_{j}\] 
for some $a_{ij} \in k$. We then define $\psi$ as follows (along with extension by linearity):
\begin{align*}
    \psi: V \otimes k[T] &\longrightarrow \ker{(q_r)} \\
     v_i \otimes \mu_j &\longmapsto \sum_{p_k \in \ker(q_r)_{\mu_j}}-a_{ik}p_k \\
    v \otimes f &\longmapsto 0 \text{ for all } f \in k[T] \, \big\backslash\, \text{span}\{\mu_i\}, v \in V.
\end{align*}
We claim that $\psi$ has the same image under $f_r -g_r$ as $\widetilde{\varphi}$. Firstly, $\displaystyle f_r(\psi)(v_i) = \psi\left(\lambda_i \otimes v_i\right) = 0$ for all $i$ as each $\lambda_i$ is $p^{r-1}$-restricted, so can't lie in $\text{span}\{\mu_j\}$, as from above we know the weights $\mu_j$ are not $p^{r-1}$ restricted. Next, $\displaystyle -g_r(\psi)(v_i) = \theta_{P_{\lambda,r}}\Big(\mu_j \longmapsto \hspace{-3ex}\sum_{p_k \in \ker(q_r)_{\mu_j}}\hspace{-3ex}a_{ik}p_k$ for each $j\Big)$. (Note that the minus signs in the definition of $\psi$ and in $-g_r$ cancel out.) 

Now,  recall that $\theta_{P_{\lambda,r}}: \Hom_k(k[G], P_{\lambda,r}) \longrightarrow P_{\lambda,r}$ is equivalent, after the identification $\Hom_k(k[G], P_{\lambda,r}) \cong k[G]^* \otimes P_{\lambda,r}$, to the composition \[\textup{eval}  \circ (\text{Id}_{k[G]^*} \otimes \Delta_{P_{\lambda,r}}): k[G]^* \otimes P_{\lambda,r} \longrightarrow k[G]^* \otimes P_{\lambda,r} \otimes k[G] \longrightarrow P_{\lambda,r}\] where $\textup{eval}$ is the obvious evaluation map. Thus, we have $\displaystyle -g_r(\psi)(v_i) = \sum_{j} a_{ij}p_{j}$, as required. Finally, observe that $q_{r_*}(\psi) = q_r \circ \psi \equiv 0$, by construction, thus $\widetilde{\varphi}-\psi \in \ker(f_r-g_r)$ with $q_{r_*}(\widetilde{\varphi}-\psi) = \varphi$. Therefore, $q_{r_*}: \ker(f_r-g_r) \longrightarrow \ker(f_{r-1}-g_{r-1})$ is surjective for large enough $r$. In particular, $\{\ker(f_r-g_r)\}$ is Mittag-Leffler and the proof is complete.
\end{proof}

We are now in a position to state and prove the main result. First, recall that a projective cover of an object $M$ in a category $C$ is a pair $(P,\phi)$ where $P$ is a projective object in $C$, and $\phi$ is a superfluous homomorphism. In abelian categories, in particular in $k[G]$-Contra, this is equivalent to $\ker(\phi)$ being a superfluous subobject of $P$. That is, if $H \subseteq P$ is a subobject of $P$ with $H + \ker(\phi) = P$, then $H  = P$

\begin{theorem}\label{P_lambda_main_thm}
Let $G$ be a simply connected, semisimple algebraic group over an algebraically closed field $k$ of characteristic $p$ and let $\lambda \in X_+(T)$ be a dominant weight.

Assume that $G$-structures $P(\lambda)$ exist on $P_1(\lambda)$, the $G_1$-projective covers of $L_1(\lambda)$, for each p-restricted weight $\lambda \in X_1(T)$ (see Remark \ref{HumVermaconjecture}).

Then the $k[G]$-contramodule $P_{\lambda}$ is the projective cover of $L(\lambda)$ in the category $k[G]$-Contra.
\end{theorem}

\begin{proof}
Write $\displaystyle \lambda = \sum_{i=0}^sp^i\lambda_i$, with each $\lambda_i \in X_1(T)$ a $p$-restricted weight. Since $P_{\lambda} = \varprojlim P_{\lambda,r}$ and $\displaystyle P_{\lambda,r}\big/\textup{Rad}_{G_r}(P_{\lambda,r})$ is isomorphic to  $L_r(\lambda)$ as $G_r$-modules (or equivalently $k[G_r]$-contramodules) for each $r > s$, we have that ${\displaystyle {P_{\lambda}\, \big/\textup{ Rad}_{G}(P_{\lambda}) \cong L(\lambda)}}$ in the category $k[G]$-Contra, where we have defined $\textup{Rad}_{G}(P_{\lambda}) := \varprojlim\textup{Rad}_{G_r}(P_{\lambda,r})$. Now, suppose that $M \subset P_{\lambda}$ is a subcontramodule such that $M + \textup{Rad}_{G}(P_{\lambda}) = P_{\lambda}$. Applying the natural projection $\pi_r$ gives $\pi_r(M) + \textup{Rad}_{G_r}(P_{\lambda,r}) = P_{\lambda,r}$. This implies that $\pi_r(M) = P_{\lambda,r}$ for all $r$, since each $\textup{Rad}_{G_r}(P_{\lambda,r})$ is superfluous in $P_{\lambda,r}$. Therefore we have $M = P_{\lambda}$ and so $\textup{Rad}_{G}(P_{\lambda})$ is superfluous in $P_{\lambda}$. 

All that remains is to show that $P_{\lambda}$ is a projective $k[G]$-contramodule. By Lemma \ref{local_finiteness} we only have to show that $\text{Cohom}_{k[G]}( -,P_{\lambda})$ is exact for finite-dimensional comodules. However, by Lemma \ref{lim_commutes_with_cohom} we have, for a given finite-dimensional comodule $V$,
\[\textup{Cohom}_{k[G]}(V,P_{\lambda})\, =\, \varprojlim \textup{Cohom}_{k[G]}(V,P_{\lambda,r}).\]
Now, we know by Lemma \ref{duality_lemma} that for large enough $r$, $\textup{Cohom}_{k[G]}(V,P_{\lambda,r})$ has dimension $f(V)$, thus, \[\textup{dim}\big(\textup{Cohom}_{k[G]}(V, P_{\lambda})\big) = f(V).\]
Finally, as $f$, the composition multiplicity of $L(\lambda)$, is additive on short exact sequences of finite-dimensional comodules we conclude that $\textup{Cohom}_{k[G]}( -,P_{\lambda})$ is exact on (short exact sequences of) finite-dimensional comodules, thus $P_{\lambda}$ is projective.
\end{proof}


The existence of projective covers of $G$-modules (considered as $k[G]$-contramodules) may also be recovered from the inverse limit theorem along with the ``dual functor" from $k[G]$-comodules to $k[G]$-contramodules. Specifically, the result is as follows:
\begin{theorem}\label{Donkinconj}
Let $G$ be a simply connected semisimple algebraic group. Let $I(\lambda)$ be the rationally injective $G$-envelope of $L(\lambda)$. Suppose $I(\lambda) = \varinjlim_i X_i$, where the $X_i$ are finite-dimensional $G$-modules. Then $P(\lambda^*)$, the $k[G]$-Contra projective cover of $L(\lambda)^*$, the dual module of $L(\lambda)$, is isomorphic to $\varprojlim_i X_i^*$, as contramodules.    
\end{theorem}

Observe that if $L(\lambda)$ is a left module, then $L(\lambda)^*$ is naturally a right module. We will not convert from right to left via the use of the inverse $g \mapsto g^{-1}$, and so for the rest of this section the $L(\lambda)^*$ should be considered as left contramodules, and so the projective covers we construct should be left contramodules. Of course, all previous results hold with a careful modification of lefts and rights, and anyway, one may use the antipode in $k[G]$ to work on the right if one wishes. First, before proving the theorem, we give two lemmas which will help with the proof. 

\begin{lemma}\label{direct inverse limit duality}
Let $(V_i,\alpha_{ij})$ be a direct system of finite-dimensional right $k[G]$-comodules. Then $(V_i^*, \alpha_{ij}^*)$ is an inverse system of left $k[G]$-contramodules. Moreover, 
$ \displaystyle \Big(\varinjlim_i V_i \Big)^* \cong \varprojlim_i V_i^*$ as left $k[G]$-contramodules.
\end{lemma}

\begin{proof} The first part follows by applying the functor $\Hom_k(-,k) : k[G]$-Comod $\rightarrow k[G]$-Contra$_\textit{left}$ as discussed in Section \ref{Section-contra-defns} (with lefts and rights swapped).

For the last part, we have $\displaystyle \varinjlim_i V_i = \Big(\bigoplus_i V_i\Big)\Big/{\big(v_i - \alpha_{ij}(v_i), j > i, v_i \in V_i\big)}$. We define the kernel $\displaystyle (K,\kappa) = \ker\big(\bigoplus_i V_i \longrightarrow \varinjlim_i V_i \big)$, then we have a short exact sequence of right $k[G]$-comodules
\[0 \longrightarrow K \xrightarrow{\kappa} \bigoplus_i V_i \longrightarrow \varinjlim_i V_i \rightarrow 0.\]
The functor $(-)^* = \Hom(-,k): k[G]\text{-Comod} \longrightarrow k[G]\text{-Contra}_\textit{left}$ is exact. Applying it gives a short exact sequence of left $k[G]$-contramodules
\[ 0 \longrightarrow \big(\varinjlim_i V_i \big)^* \longrightarrow \prod_i V_i^* \xrightarrow{\kappa^*} K^* \longrightarrow 0.\]
Thus, we see that $\displaystyle \big(\varinjlim_i V_i \big)^* = \ker\Big(\prod_i V_i^* \xrightarrow{\kappa^*} K^*\Big).$ On the other hand, suppose $(f_k) \in \ker(\kappa^*)$, then $\big(\kappa^*(f_k)\big)\big(v_i - \alpha_{ij}(v_i)\big) = 0$ for all $v_i \in V_i$ for all $j \geq i$, or equivalently, $f_i = \alpha_{ij}^*f_j$ for all $j \geq i$. Therefore we have
\[\displaystyle \big(\varinjlim_i V_i \big)^* = \left\{(f_k) \in \prod_i V_i^* : f_i = \alpha_{ij}^*(f_j) \text{ for all } j > i  \right\} = \varprojlim_i \left( V_i^*, \alpha_{ij}^* \right).\]
\end{proof}

\begin{lemma}\label{inj co proj contra}
Let $I$ be an injective right $k[G]$-comodule. Then $\Hom_k(I,k)$ is a projective left $k[G]$-contramodule
\end{lemma}

\begin{proof}
Recall, that as $I$ is injective, the coaction map $I \longrightarrow I \otimes k[G]$ splits. Applying the functor $\Hom_k(-,k): \text{Comod-}k[G] \longrightarrow k[G]\text{-Contra}_{\textit{left}}$ to the coaction map gives a split map of (left) contramodules $\Hom_k(I \otimes k[G],k) \longrightarrow \Hom_k(I,k)$. 

However, observe that we have $\Hom_k(I \otimes k[G],k) \cong \Hom_k\big(k[G],\Hom_k(I,k)\big)$ as contramodules, which is free. Thus, $\Hom_k(I,k)$ is a direct summand of a free contramodule, hence projective.

\end{proof}

We can now give a proof of Theorem \ref{Donkinconj}.

\begin{proof}
To begin, $I(\lambda)$ is an injective right $k[G]$-comodule, and so Lemma \ref{inj co proj contra} implies that $I(\lambda)^*$ is a projective left $k[G]$-contramodule. To determine the radical of $I(\lambda)^*$ one may have hopes of using an analog of the fact that for a (Noetherian) ring $R$ and finitely generated $R$-module $M$ we have $\text{rad}(M^*) = \big(M/\text{soc}(M)\big)^*$. However, unlike the case of finitely generated modules, the functor $\Hom_k(-,k): k[G]$-Comod$^\text{op}\longrightarrow k[G]$-Contra$_\textit{left}$ is not a duality, and so we do not get the result immediately from this. 

Instead, recall that $I(\lambda) = \varinjlim_i X_i$, where each $X_i$ is a finite-dimensional subcomodule of $I(\lambda)$. For each $i$ we have a short exact sequence
\[0 \longrightarrow \textup{soc}(X_i) \longrightarrow X_i \longrightarrow X_i\big/\textup{soc}(X_i) \longrightarrow 0. \]
Now, since each $X_i$ is a subcomodule of $I(\lambda)$ we must have $\text{soc}(X_i) = {0}$ or $\text{soc}(X_i) = L(\lambda)$. It follows immediately that $\varinjlim_i \text{soc}(X_i) = \text{soc}(\varinjlim_i X_i) = \text{soc}(I(\lambda)) = L(\lambda)$, as colimits commute with colimits. Lastly taking direct limits of $k[G]$-comodules is exact. Altogether, this gives us an exact sequence
\[0 \longrightarrow L(\lambda) \longrightarrow I(\lambda) \longrightarrow \varinjlim_i \Big(X_i \big/ \textup{soc}(X_i)\Big) \longrightarrow 0. \]
Next, applying $\Hom_k(-,k):k[G]\text{-Comod}^{\text{op}} \longrightarrow k[G]\text{-Contra}_\textit{left}$ along with Lemma \ref{direct inverse limit duality} yields
\[ 0 \longrightarrow \varprojlim_i \Big(X_i \big/ \textup{soc}(X_i)\Big)^* \longrightarrow I(\lambda)^* \longrightarrow L(\lambda)^* \longrightarrow 0.\]
Now, the functor $\Hom_k(-,k):k[G]$-Comod$_{\textit{f.d.}}^\text{op} \longrightarrow k[G]$-Contra$_{\textit{left, f.d.}} \textit{ is }$a duality. Since the $X_i$ are finite-dimensional we have $\displaystyle \big(X_i \big/ \textup{soc}(X_i)\big)^* \cong \textup{rad}(X_i^*)$. Once again, since the $X_i$ are subcomodules of $I(\lambda)$ we must have that $\textup{rad}(X_i^*) = X_i^*$ or $\textup{rad}(X_i^*) \cong \big(X_i \big/ L(\lambda)\big)^*$ and we deduce that $\varprojlim_i \textup{rad}(X_i^*) = \textup{rad}(\varprojlim_i X_i^*)$, as limits commute with limits. Finally, lemma \ref{direct inverse limit duality} gives us that $ \varprojlim_i X_i^* = \Big(\varinjlim_i X_i\Big)^* = I(\lambda)^*$. Plugging this into the above sequence yields 
\[0 \longrightarrow \textup{rad}\big(I(\lambda)^*\big) \longrightarrow I(\lambda)^* \longrightarrow L(\lambda)^* \longrightarrow 0.\]
Thus, $I(\lambda)^*$ is the projective cover of $L(\lambda)^*$ in $k[G]$-Contra$_\textit{left}$.
\end{proof}

\printbibliography

\end{document}